\documentclass[reqno,12pt]{amsart}
\usepackage{amsmath,accents}
\usepackage{amsthm}
\usepackage{mathtools}

\usepackage[left=2.5cm,right=2.5cm,top=2.5cm,bottom=2.5cm]{geometry}
\usepackage{enumitem}

\usepackage[colorlinks=true, pdfstartview=FitV, linkcolor=blue, 
            citecolor=blue, urlcolor=blue]{hyperref}
\usepackage{todonotes}

\makeatletter
\@namedef{subjclassname@2020}{\textup{2020} Mathematics Subject Classification}
\makeatother

\numberwithin{equation}{section}
\theoremstyle{plain}
\newtheorem{theorem}{Theorem}[section]
\newtheorem{proposition}[theorem]{Proposition}
\newtheorem{lemma}[theorem]{Lemma}

\newtheorem{definition}[theorem]{Definition}

\newcommand{\CC}{\mathbb{C}} 
\newcommand\epsi{\varepsilon}
\newcommand\ove{\overline}
\newcommand\haz{\widehat}
\newcommand{\Rz}{{\mathbb R}}
\newcommand{\Nz}{{\mathbb N}}

\renewcommand{\d}{{\rm d}}
\DeclareMathOperator{\esssup}{{\rm ess\,sup}} 

\newcommand\BBB{\color{black}}
\newcommand\EEE{\color{black}}

\title[Accretive growth in elastic solids] {An existence result
  for \\ accretive growth in elastic solids}
\author[E. Davoli] {Elisa Davoli} 
\address[Elisa Davoli]{Institute of Analysis and Scientific Computing, TU Wien, 
Wiedner Hauptstrasse 8-10, A-1040 Vienna, Austria}
\email{elisa.davoli@tuwien.ac.at}
\author[K. Nik] {Katerina Nik}
\address[Katerina Nik]{University of Vienna,  Faculty of Mathematics, 
Oskar-Morgenstern-Platz 1, A-1090 Vienna, Austria}
\email{katerina.nik@univie.ac.at}
\author[U. Stefanelli]{Ulisse Stefanelli} 
	\address[Ulisse Stefanelli]{University of
		Vienna, Faculty of Mathematics,
                Oskar-Morgenstern-Platz 1, A-1090 Vienna, Austria. \\
		University of Vienna, Vienna Research Platform on Accelerating
		Photoreaction Discovery, W\"ahringerstra\ss e 17, 1090 Wien, Austria.\\
	 Istituto di	Matematica Applicata e Tecnologie Informatiche {\it E. Magenes}, via
		Ferrata 1, I-27100 Pavia, Italy
	}
	\email{ulisse.stefanelli@univie.ac.at}
	\urladdr{http://www.mat.univie.ac.at/$\sim$stefanelli}
\author[G. Tomassetti]{Giuseppe Tomassetti}
\address[Giuseppe Tomassetti]{Dipartimento di Ingegneria, Universit\`a
  degli Studi Roma Tre, via Volterra 62, I-00146 Roma, Italy}
\email{giuseppe.tomassetti@uniroma3.it}

 %Viscosity solutions to Hamilton-Jacobi equations in optimal control and differential games


\begin{document}

\begin{abstract} 
 We  investigate a model for the accretive growth of an 
 elastic  solid.   The reference configuration of the body is
 accreted in its normal direction, with space- and deformation-dependent
 accretion rate. The time-dependent reference configuration is identified via
 the level sets of the unique viscosity solution of a suitable
 generalized eikonal equation. After proving the global-in-time
 well-posedness of the quasistatic equilibrium %problem when the growth of the body
  under prescribed growth,  we prove  the existence of a local-in-time 
 solution for the coupled equilibrium-growth  problem,  where 
 both mechanical displacement and time-evolving set are unknown. %s of
 % the problem.
 A distinctive challenge is the limited  %key challenge of the
 % latter setting is the
 regularity of the growing body,  which calls for proving a new
 uniform Korn inequality. % This is overcome by means of a level-set approach allowing to deal with possible singularities originating from self-touching of the growing body. 
\end{abstract}
\maketitle

\section{Introduction}

%%%%%Old verson:

%A wealth of mechanical systems experience
%{\it accretive growth}, namely, the progressive accretion of the
%system by addition of mass at its boundary. 
%This is paramount to many
%  biological systems, which commonly change shape and function over time.
%In addition, accretive growth is a key aspect in a variety of
%technological applications, ranging from coating to additive
%manufacturing \cite{goriely}.
%%%%%%%

Many mechanical systems experience
{\it accretive growth}, namely, progressive growth 
of  a body  by adding mass at its boundary. This paradigm is of paramount relevance 
to numerous  biological systems, where shape and function evolve over
time:  The formation of horns, teeth and seashells
\cite{skalak,thompson}, secondary growth in trees
\cite{fournier-et-al}, and cell motility  due to actin growth  \cite{hodge} are
examples of accretive growth in nature. 
 Furthermore,  accretive growth is a key aspect in a variety of
technological applications,  such as, for example, metal solidification
\cite{schwerdtfeger-et-al}, crystal growth \cite{langer},  and additive manufacturing~\cite{horn}. 

The first theoretical study of accretive growth involved  the
analysis of  thick-walled cylinders manufactured by wire winding
of an initial elastic tube \cite{southwell}. Within the framework of
linear elasticity, one of the earliest problems addressed was that of
a growing planet subject to self-gravity \cite{brown}. In
\cite{metlov}, the author proposed  a  first
large-deformation theory of accretion, specifically tailored for aging
viscoelastic solids undergoing accretion.  This work also introduced
the notion of {\it time-of-attachment map},  which is the function
$\theta$ used here, as well.  The engineering literature on
accretive growth is vast. Among the many contributions, we single out
\cite{abi,bergel} for the modeling of surface growth on deformable
substrates, \cite{epstein} for a description of the kinetics of
boundary growth, \cite{lychev,truskinovsky,zurlo} for studies
 in  the setting of nonlinear elasticity\BBB, as well as \cite{tomassetti}
for a description of accretion on a hard spherical surface.  A
finite strain model combining accretion and ablation may also be found in 
\cite{sozio}.  For more detailed reviews, we refer to \cite{drozdov,naumov, sozioy}.

Compared with the extensive engineering literature, rigorous
mathematical results on the mechanics of growth, whether accretive or
volumetric,  are sparse \cite{bressan, morpho}. To the best of our
knowledge, the most significant mathematical efforts thus far have
primarily focused  either on performing numerical simulations or on
specifying the correct modeling framework for tailored applications, see, e.g., \cite{barreira-et-at,egberts-et-al,ganghoffer-et-al}.

%%%Old version: 
%Compared with the vast engineering literature, rigorous mathematical
%results on growth mechanics are scant \cite{bressan, morpho}. To the best of our knowledge, the biggest mathematical efforts have so far mostly been addressed to specifying the mechanical assumptions and performing numerical simulations \dots
In this paper, we revisit an accretive-growth model advanced by
{\sc  Zurlo and 
  Truskinovski} \cite{zurlo}.
 Accretive growth  is described by specifying a set-valued
time-dependent function $t \in [0,T] \mapsto \Omega(t) \subset
\Rz^d \, (d\in \Nz)$, identifying
at each time $t$ the reference configuration of the body under study. 
Such function is
increasing in time with respect to set inclusion (growth) and has 
open and connected values. Assume for the time being that $\Omega(t)$
is known (note however that this will also be an unknown later
on). One can equivalently describe the evolution of the
time-dependent reference configuration by means of a 
time-of-attachment  continuous function $\theta:
\Rz^d\to [0,\infty)$, whose value $\theta(x)$ indicates the time at
which the point $x$ is {\it added} to the solid.  Correspondingly, one
defines $\Omega(0):={\rm int} \{x\in \Rz^d\ : \ \theta(x)=0\}$
(interior part)
and 
$$
\Omega(t) := \{x \in \Rz^d \ : \ \theta(x)<t\} \quad \text{for all } t>0. 
$$
Let us stress from the very beginning that the growth process may prevent $\Omega(t)$ from being 
regular at specific times, posing a challenge to the analysis.  
The mechanical state of the body is then described by its
{\it displacement} $u(\cdot,t):\Omega(t) \to \Rz^d$ from the
time-dependent reference configuration $\Omega(t)$.

Growth and mechanical equilibration generally occur on very distinct
time scales. The usual time frame for growth ranges between minutes and
months, whereas mechanical equilibrations can take up to few milliseconds,
depending on the material. This basic observation leads us to
consider the equilibrium of the growing object in its quasistatic
approximation, namely,
\begin{equation}
  \label{eq:quasi}
-\nabla \cdot \sigma (x,t) = f(x,t)  \quad \text{for} \ x \in
\Omega(t),  \ t \in [0,T]
\end{equation}
 where $\sigma$ indicated the {\it stress}. 
We assume a linear elastic response in the solid. As
accretive growth is known to  generate residual stresses,
 following \cite{zurlo} we postulate
the constitutive relation
\begin{equation}
  \label{eq:sigma}
  \sigma(x,t) = \CC \big(\epsi(u(x,t)) - \alpha(x)\big)  \quad \text{for} \ x \in
\Omega(t),  \ t \in [0,T].
\end{equation}
Here, $\CC$ is the symmetric and positive-definite {\it elasticity
  tensor}, the symmetrized displacement gradient $\epsi(u) = (\nabla u
+ \nabla u^\top)/2$ is the {\it strain}, and $\alpha:\Omega(T) \to
\Rz^{d\times d}$ denotes the {\it backstrain}, which has been
accumulated during growth. By assuming that material is added at the
boundary of the solid in a locally unstressed state, we would follow 
\cite{zurlo} and postulate
$$
  \alpha(x)= \epsi(u)(x,\theta(x))\quad \text{for} \ x \in \Omega(T). 
  $$
  In fact, together with \eqref{eq:sigma} this would entail that
  $\sigma(x,\theta(x))=0$. On the other hand, such %Such a
  position would require $\epsi(u)$ to admit a space-time trace at
points of the form $(x,\theta(x))$, a possibility which might be
impeded by the low regularity of $\Omega(t)$. We hence resort to a
mollification of the above position  by actually defining
\begin{equation}
  \alpha(x):= (K\ove \epsi)(u)(x,\theta(x))\quad \text{for} \ x \in \Omega(T),\label{eq:alpha}
\end{equation}
where $\ove \epsi(u)(\cdot,t)$ denotes the trivial extension of
$\epsi(u) (\cdot,t)$ to the whole $\Rz^d$ and $K$ is a space-time
convolution operator of the form
\begin{equation} (K\ove\epsi(u))(x,t):=\int_0^t\!\!\int_{\Rz^d}k(t-s)\phi(x-y)\ove\epsi(u)(y,s)\,
\d y\,  \d s\label{eq:K}
\end{equation}
for given time- and space-kernels $k \in  W^{1,1}(0,T)$ and $\phi\in
H^1(\Rz^d)$, respectively. From the modeling viewpoint, definition
\eqref{eq:alpha} links the residual growth-originated backstrain $\alpha(x)$ to
the {\it local mean} strain state at added material points, rather than to
the {\it pointwise} one. By choosing the supports of $k$ and $\phi$
sufficiently small around $0$, one has the possibility of arbitrarily localizing
this effect. Still, under the action of the (trivial
extension and) convolution operator one is allowed to take trace
values on the manifold $(x,\theta(x))$, without further regularity
restrictions.

Ideally, we would complement the equilibrium system \eqref{eq:quasi} by traction-free
boundary conditions
$$
\sigma(x,t)\,n(x,t) =0 \quad \text{for} \ x \in
\partial \Omega(t),  \ t \in [0,T].
$$
Still, as the growing set $\Omega(t)$ cannot be expected to be regular for
all times, a classical Sobolev trace on $\partial
\Omega(t)$ might be not available for some $t$.  Hence, the latter {\it natural}
condition will have to be casted variationally, within a weak reformulation
of \eqref{eq:quasi}-\eqref{eq:K}, see
\eqref{eq:eq1} below. To this
aim, in order to filter out rigid-body motions, some condition has to be
added to the equilibrium system \eqref{eq:quasi}-\eqref{eq:alpha}. As
the boundary $\partial \Omega(t)$ is evolving, in order to keep
notation to a minimum we ask for the {\it docking condition}
\begin{equation}
  \label{eq:boundary}
  u (x,t)= 0 \quad \text{on} \ \ \omega \times [0,T],
\end{equation}
where we have  fixed the  docking set $\omega\subset
\Omega(0)$. Condition \eqref{eq:boundary} bears some applicative
relevance, especially for $d=1$ or $2$. %Other conditions could be
                                %considered as well, possibly at the
                                %expense of additional notational complications.

Let us now turn our attention to the accretion process. Here, we
intend to model a situation where accretion results from deposition at
the boundary, at a given rate. Correspondingly, a point $x(t)\in
\partial \Omega (t)$ at the
boundary is assumed to follow the {\it normal accretion} law
$$ \dot x(t) = \gamma\, n (x(t))$$
where $n(x(t))$ indicates the outward normal to $\partial \Omega (t)$ at
$x(t)$ and $ \gamma$ is the growth rate, which will be later assumed
to be dependent on position and strain. Note that the evolution of
$\Omega(t)$ depends on its intrinsic geometry via $n(x(t))$. As the
level sets of  the 
function $\theta $ correspond by definition to the sets $\Omega(t)$,
one formally has that $n(x(t)) = \tfrac{\nabla \theta(x(t))}{|\nabla
\theta(x(t))|}$. At the same time, by differentiating the equality
$t=\theta(x(t)) $ with respect to time, one gets 
$$
1 = \nabla
\theta(x(t)) \cdot \dot x(t) = \nabla
\theta(x(t)) \cdot \gamma\, n (x(t))= \gamma \, \nabla
\theta(x(t)) \cdot \frac{\nabla
\theta(x(t))}{|\nabla
\theta(x(t))|} = \gamma |\nabla
\theta(x(t))|
$$
 so that $\theta$ ultimately solves the eikonal equation
$\gamma|\nabla\theta|=1$. 
Growth processes are known to be inhomogeneous and to be
dependent on the deformation state~\cite{goriely}. We model this by
letting the growth rate $\gamma$ depend smoothly on the point $x(t)$ and
the strain at $x(t)$. Note that no dependence on the stress is directly
accounted for by this model. In fact, in the nonregularized case $\alpha(x(t))= \epsi(u)(x(t),t)$, position \eqref{eq:sigma}
would imply
$\sigma(x(t), t) = 0$ at $x(t)\in \partial \Omega(t)$. Additional
dependencies of the growth rate $\gamma$ on time and displacement
could also be considered, at the cost of
 minor albeit tedious changes. 

Starting from the datum $\theta=0$ on $\Omega(0)$, the evolution of
$\Omega(t)$ is hence determined by solving the generalized eikonal equation
\begin{equation}
  \label{eq:eikonal}
 \gamma(x,\alpha(x)) |\nabla \theta(x)|=1.
\end{equation}
This equation is in principle to be solved on $\Omega(T)$ only. Still,
as this set depends on the solution $\theta$ itself, one may
conveniently solve \eqref{eq:eikonal} is some larger set containing
$\Omega(T)$ (recall that
$\alpha$ from \eqref{eq:alpha} is actually defined everywhere in
$\Rz^d$). Equation \eqref{eq:eikonal}  does not admit classical
solutions. Moreover, strong solutions of \eqref{eq:eikonal} are not
unique. We hence resort to the viscosity-solution setting, where
equation \eqref{eq:eikonal}  turns out to be well-posed. Note that the continuity of $x\mapsto \gamma(x,\alpha(x))$
is needed in order to tackle problem \eqref{eq:eikonal} in the 
setting of viscosity
solutions. Such continuity calls for some smoothness of $\alpha$, which is
in turn guaranteed by our positions \eqref{eq:alpha}-\eqref{eq:K}.

The main result of the paper, Theorem \ref{thm:coupled}, provides the
existence of a  weak local-in-time solution to the coupled equilibrium-growth problem
\eqref{eq:quasi}-\eqref{eq:eikonal}. Note that our level-set
formulation via $\theta$ allows us to consider the evolution problem
beyond singularities, which occur as the growing body self-touches.
As an intermediate  step toward  Theorem \ref{thm:coupled} we discuss the 
 global well-posedness %of the pure growth process \eqref{eq:eikonal} for
% given $\alpha$ in Proposition \ref{prop:eikonal} and
of the equilibrium problem
\eqref{eq:quasi}-\eqref{eq:boundary} for given $\theta$, see Theorem \ref{prop:equi}.
Compared with the analysis in \cite{zurlo}, the novelty of our result
is twofold. At first, we do not assume to know the displacement at
the added material point. Secondly, we do not assume the
evolution $t \mapsto \Omega(t)$ to be known, but rather solve for it,
taking into account mechanical couplings and the possible onset of
singularities.

 The paper is organized as follows: in Section \ref{sec:setting-etc} we specify our assumptions and state our main results. Section \ref{sec:equi} is devoted to the proof of Theorem \ref{prop:equi}, whereas Theorem \ref{thm:coupled} is proven in Section \ref{sec:coupled}. Eventually, in Section \ref{sec:korn} we prove a uniform Korn inequality for the class of sets generated by our growth process.

\section{Setting and main results}
\label{sec:setting-etc}

We devote this section to the specification of the problems under
scrutiny.  In particular, we introduce the assumptions and discuss
some preliminary remarks.   %, the discussion of the assumptions,
% and some preliminary remark.
The statements of our main results are in  Subsection 
\ref{sec:statements} below. 
We start by collecting  % preliminary collect
some notation which will be used throughout the paper. 

\subsection*{Notation} 
Let $d\in \Nz$. We indicate by $B_r(x_0):=\{x \in \Rz^d \ : \ |x-x_0|<r\}$ the open ball in $\Rz^d$ centered
in $x_0\in \Rz^d$ with
radius $r>0$.  By $C^0_{b}(\Rz^d;\Rz^{d\times d})$ we denote the space of bounded continuous functions on $\Rz^d$ with values in $\Rz^{d\times d}$. %, equipped with the standard supremum norm. 
 The $d$-dimensional Lebesgue measure of a 
measurable set $\Omega$
in $\Rz^d$ is denoted by $|\Omega|$. 
The symbol $a\cdot b$ classically indicates the scalar product between the two
vectors $a,\, b \in \Rz^d$. The contraction between 2-tensors $A,\, B\in \Rz^{d\times d}$ is
denoted by $A:B = A_{ij}B_{ij}$, where repeated indices are tacitly
summed over. Given the 4-tensor $\mathbb{C}\in \Rz^{d\times d\times
d \times d}$, we let 
$(\mathbb{C} : A)_{ij} = \mathbb{C}_{ijkl}A_{kl}$. We indicate by
${\mathbb I}$ the identity 4-tensor. The distance between $x\in \Rz^d$ and
the nonempty set $U\subset \Rz^d$ is denoted by ${\rm dist}\,(x;U)=\inf_{u\in
  U}|x-u|$. The  same notation is also used for the Hausdorff distance
between two nonempty sets $U_1,\, U_2 \subset \Rz^d$, namely,
${\rm dist}\, (U_1;U_2)=\max\{\sup_{u_1\in U_1}{\rm dist}\, (u_1;U_2), \sup_{u_2\in
  U_2}{\rm dist}\, (U_1;u_2),\}$. We say that $U_n \to U$ {\it in the Hausdorff sense}
iff ${\rm dist}\, (U_n;U)\to 0$.

\subsection{Assumptions  and notion of weak
  solution}\label{sec:assumptions}
 In this section, we present our assumptions on data and introduce
the notion of weak solution to problem
\eqref{eq:quasi}-\eqref{eq:eikonal}. Let us first recall the definition of a John domain. 

\begin{definition}[John domain]
 A nonempty open set $U \subset \Rz^d$ is said to be a 
\emph{John domain} if there exists a specific  point $x_0 \in U$ and a 
\emph{John constant} $C_J \in (0,1]$ such that for all points $x\in U$ one can find an
arc-length parametrized
curve $\rho: [0,L_\rho] \to U$ with $\rho(0)=x$,
  $\rho(L_\rho) = x_0$, and ${\rm dist} \,(\rho(s); \partial U) \geq
  C_J  s$ for all $s\in [0,L_\rho]$.
\end{definition}

John domains have been introduced in \cite{JohnF},  see
Figure \ref{john_figure}. 
Their name has been proposed in \cite{martio}. Note that John
domains are connected and their boundary is negligible \cite[Corollary 2.3]{koskela-rohde}. All Lipschitz domains are John, 
whereas John domains may have fractal boundaries or internal
cusps. External cusps are nonetheless excluded. 
%For instance the region inside the von Koch snowflake is a John domain.
We refer to \cite{hajlasz-koskela} and the references therein for an
overview on some important features of John domains. 
\begin{figure}[h]\centering
  \pgfdeclareimage[width=85mm]{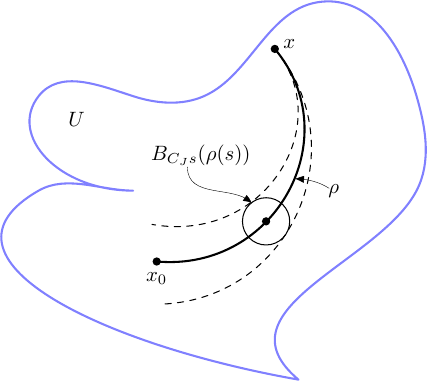}{john_figure}
  \begin{center}
    \pgfuseimage{john_figure}
  \end{center}
  \caption{A John domain in $\Rz^2$}
  \label{john_figure}
\end{figure}

% Here John refers to \cite{JohnF} who used this condition to study problems in elasticity theory; Martio and Sarvas \cite{martio} introduced this terminology. The definition says, roughly speaking, that each point in a John domain can be connected to the central point without getting too close to the boundary. 

% Note that John domains are connected. Note also that the requirements for being a John domain are much weaker compared to those for being Lipschitz domains. In particular, every Lipschitz domain is a John domain and John domains may have fractal boundaries or internal cusps, while the formation of external cusps is excluded. For instance the region inside the von Koch snowflake is a John domain.

% We refer to \cite{hajlasz-koskela} and the references therein for an overview on some important features of John domains. \EEE
% \MMM\texttt{E: we might add a picture here}\EEE\\

%We assume the following to hold:
\noindent The following assumptions will be used throughout the  paper
without further mention:
\begin{align}
&T>0, \label{a:T}\tag{A1}\\
&\label{a:TC} \CC \in \Rz^{d\times d \times d \times d}\ \ \text{is
                              symmetric with} \ \  \CC \geq c_*
                              {\mathbb I} \ \ \text{for
                               some} \ c_*>0,\tag{A2}\\
  & \label{a:Omega} \Omega(0) \subset \Rz^d  \ \ \text{is a bounded
    John domain},\tag{A3} \\
  & \label{a:omega}  \omega \subset \subset \Omega(0) \ \ \text{is
    nonempty, open, and connected and} \ \ {\rm dist}\,(\omega;
    \partial \Omega(0)) =:\rho_0>0,  \tag{A4} \\
  & \label{a:gamma} \gamma\in C^{ 0 \EEE,1}(\Rz^d \times \Rz^{d
    \times d} ; \Rz) \ \ \text{with} \ \ \gamma_* \leq \gamma (\cdot)
  \leq\gamma^* \ \ \text{for some} \ \  0<\gamma_* <\gamma^*, \tag{A5}\\
  & \label{a:K} k \in  W^{1,1}(0,T), \  \ \phi\in  H^1 
    (\Rz^d) \ \ \text{with compact  support}, \tag{A6}  \\
  & \label{a:f} f \in  C^0([0,T] \EEE ;L^2(\Rz^d;\Rz^d)),\tag{A7}
\end{align}
%%%%%
%Old version via align : 
%\begin{align}
%  & T>0, \label{a:T}\\[0.1cm]
%&\CC \in \Rz^{d\times d \times d \times d} \
%    \text{symmetric with $\CC \geq c_* {\mathbb I}$ for $c_*>0$,
 % where ${\mathbb I}$ is the identity tensor},\label{a:TC}\\[0.1cm]
%  &   \Omega(0) \subset \Rz^d \ \ \text{bounded  John domain\EEE}, \label{a:Omega}\\[0.1cm]
%  & \omega \subset \subset \Omega(0) \ \ \text{nonempty and open}, \label{a:omega}\\[0.1cm]
%  & \gamma\in C^{ 0 \EEE,1}(\Rz^d \times \Rz^{d \times d} ; \Rz) \ \ \text{with} \ \ 0< \gamma_* \leq \gamma (\cdot) \leq
 %   \gamma^*, \label{a:gamma}\\[0.1cm]
 %% & Q \subset \Rz^d \ \ \text{bounded, closed, and connected with} \ \ \Omega(0) +   B_{ \gamma_*T}(0)\subset Q, \label{a:Q}\\
%  & k \in  W^{1,1} \EEE (0,T), \ \phi\in  H^1 \EEE (\Rz^d) \
%    \text{ with compact  support}, \label{a:K}\\[0.1cm]
%  & f \in  C^0([0,T] \EEE ;L^2(\Rz^d;\Rz^d)), \label{a:f}
%\end{align}
%%%%%%%
where the inequality in \eqref{a:TC} is meant in the sense of the L\"owner order. 
%In the rest of the paper, \eqref{a:T}-\eqref{a:f}  are tacitly
%assumed,  
%without further mention.
 Given the kernels $k$ and $\phi$ from \eqref{a:K}, we define the
space-time nonlocal operator $K$ as 
\begin{equation}
  (Ke)(x,t):=\int_0^t\!\!\int_{\Rz^d}k(t-s)\phi(x-y)e(y,s)\,
    \d y\,  \d s \qquad \text{for all } e \in L^1( \Rz^d \times (0,T) ;\Rz^{d\times d}).\label{a:K2}
\end{equation}  
%  Note that our choice 
%   \eqref{a:T}-\eqref{a:f} is motivated by the sale of simplicity,
%   rather than maximal generality. At the expense of some more involved
% arguments, \eqref{a:T}-\eqref{a:f} could be weakened.
For any nonempty open set $\Omega  \subset \subset \Rz^d$ with $\omega
\subset \Omega$ we use the notation
$$H^1_\omega (\Omega;\Rz^d)\coloneqq \{u\in H^1
(\Omega;\Rz^d):\,u=0\text{ on }\omega\}.$$
 Moreover, we indicate by $\ove \epsi$ the trivial extension to
$\Rz^d$ of a measurable function $\epsi $ defined on $\Omega$. 

The weak formulation of problem \eqref{eq:quasi}-\eqref{eq:eikonal}
reads 
\begin{align}
   &u(\cdot,t) \in  H^1_\omega(\Omega(t);\Rz^d) \ \  \text{and} \
     \ \nonumber \\
   &\qquad \int_{\Omega(t)}  \CC (\epsi (u(x,t)) -\alpha(x)):\epsi(v(x)) \,\d x =
     \int_{\Omega(t)}f(x,t)\cdot v(x)\, \d x \nonumber \\
   &\qquad
      \text{for all } v \in  H^1_\omega(\Omega(t);\Rz^d), \ \text{for
     a.e.} \ t \in (0,T),\label{eq:eq1}\\[2mm] 
 &\alpha(x) = K\ove \epsi(u) (x,\theta(x)) \quad \text{for all } x \in
                                           \Rz^d,   \label{eq:eq2}\\[2mm]
  & \gamma (x,\alpha(x)) |\nabla (-\theta)(x)| =1  \ \  \text{in the
    viscosity sense in} \
     \Rz^d\setminus \overline{\Omega(0)}, \label{eq:theta1}\\[2mm]
 & \theta=0  \quad \text{on} \ \Omega(0).\label{eq:theta2}
\end{align}
In the following, given $\alpha \in C^0_{b}(\Rz^d;\Rz^{d\times
  d})$, equation \eqref{eq:theta1} will be solved in the following
viscosity sense. 
\begin{definition}[Viscosity solution] Let $\alpha \in C^0_{b}(\Rz^d;\Rz^{d\times
  d})$ be given and $\theta :\Rz^d \to [0,\infty)$ be continuous. 
 
 We say that $\theta$ is a \emph{viscosity subsolution} of \eqref{eq:theta1} if for
all $x_0\in\Rz^d \setminus \overline{\Omega(0)}$ % with $\theta(x_0)<
% T$
and any smooth function
$\varphi$ with $\varphi(x_0)=-\theta(x_0)$ and $\varphi\geq-\theta$ in a neighborhood of $x_0$, it holds that $ \gamma(x_0,\alpha(x_0))|\nabla\varphi(x_0)|\leq 1$. 

Similarly, we say that $\theta$ is a \emph{viscosity supersolution} of \eqref{eq:theta1} if for
all $x_0\in\Rz^d\setminus \overline{\Omega(0)}$ %with $\theta(x_0)< T$
and any smooth function
$\varphi$ with $\varphi(x_0)=-\theta(x_0)$ and $\varphi\leq-\theta$ in a
neighborhood of $x_0$, it holds that $ 
 \gamma(x_0,\alpha(x_0)) |\nabla\varphi(x_0)|\geq 1$. 

 Finally, $\theta$ is said to be a \emph{viscosity solution} of 
 \eqref{eq:theta1} if it is both a viscosity sub- and supersolution.
\end{definition}

%%%%%%%%Old version:
%Recall that the continuous function $\theta$ is a viscosity
%subsolution (supersolution, respectively) of \eqref{eq:theta1} if for
%all $x_0\in\Rz^d$ with $\theta(x_0)< T$ \EEE and any smooth function
%$\varphi$ with $\varphi(x_0)=-\theta(x_0)$ and $\varphi\geq-\theta$ in a
%neighborhood of $x_0$ ($\varphi\leq-\theta$, resp.) one has 
%$ \gamma(x_0,\alpha(x_0)) \EEE |\nabla\varphi(x_0)|\leq 1$ ($ 
% \gamma(x_0,\alpha(x_0)) \EEE |\nabla\varphi(x_0)|\geq 1$, resp.),
%and that $\theta$ is a viscosity solution if it is both a
%sub- and a supersolution.
%%%%%%%%%%%%%%%%

Let us record the following fact.
\begin{lemma}
\label{lemma:set-meas} For all $\theta:\Rz^d \to [0,\infty)$
continuous the set $Q_S\coloneqq \cup_{t\in
    (0,S)} \Omega(t) \times \{t\}$ is measurable for every $S\in (0,T]$.
\end{lemma}

 \begin{proof}
 Fix $S\in (0,T]$. From the continuity, and hence the measurability of $\theta$, it follows that the extended map $\tilde{\theta}:\Rz^d\times \Rz\to\Rz$ defined as $\tilde{\theta}(x,t)\coloneqq \theta(x)$ is measurable. Analogously, the projection $\tau:\Rz^d\times \Rz\to\Rz$ given by $\tau(x,t)\coloneqq t$ is continuous, and thus measurable. The measurability of $Q_S$ follows then by observing that
 \begin{align*}
 Q_S&=\left\{(x,t)\in \Rz^d\times (0,S) :\,\theta(x)-t<0\right\}\\
 &=\left\{(x,t)\in \Rz^d\times (0,S) :\,(\tilde{\theta}-\tau)(x,t)<0\right\}\\
 &=(\tilde{\theta}-\tau)^{-1}(-\infty,0)\cap (\Rz^d\times (0,S) ).\qedhere
 \end{align*}
 \end{proof}

 Before going on, let us comment  on the two 
 subproblems \eqref{eq:eq1}-\eqref{eq:eq2} and \eqref{eq:theta1}-\eqref{eq:theta2}.  
At first, let us discuss the eikonal problem
\eqref{eq:theta1}-\eqref{eq:theta2} by assuming to be given $\alpha\in C^0_{b}(\Rz^d;\Rz^{d\times d})$. As $x\mapsto \gamma(x,\alpha(x))$ is continuous, bounded,
and well-separated from $0$,  in view of \eqref{a:gamma}, this
eikonal problem  admits a unique viscosity solution $\theta$, cf. 
\cite{Bardi,Barles}.
In fact, the solution $\theta$ of \eqref{eq:theta1}-\eqref{eq:theta2}
is Lipschitz continuous with
\begin{equation}
\label{eq:gradtheta}
  0<\frac{1}{\gamma^*} \leq |\nabla \theta | \leq \frac{1}{\gamma_*} \
  \  \text{in} \ \ \Rz^d \setminus \overline{\Omega(0)}. 
\end{equation}
%By $\Omega(T) =\{x \in \Rz^d\,:\, \theta(x)<T\}$, 
A first consequence of these inequalities is that $\Omega(T)$ is
bounded independently of $\alpha$. Indeed, one has that
$\Omega(T)
\subset \Omega(0) + B_{T \gamma^*  }(0)$, cf. \cite{john}.  
In the following, we can hence
assume to be given a fixed bounded
open set $D \subset \Rz^d$ such that
\begin{equation}
  \label{eq:omegabound}
  \Omega(t)
\subset D \quad \text{for all } t\in [0,T]
\end{equation}
for any solution of \eqref{eq:theta1}-\eqref{eq:theta2}, namely,
independently of $\alpha$. As the support of $\phi$ is compact, see
\eqref{a:K}, $D$ can be assumed to be large enough that the trivial
$\ove \epsi$
extension of $\epsi: \Omega(t) \to \Rz^{d\times d}$ can be considered
to be defined on $D$, with no loss of generality and without
introducing new notation.

Note moreover that problem \eqref{eq:theta1}-\eqref{eq:theta2}  is stable with
respect to data convergence. More precisely, if $\alpha_n \to \alpha $
 locally uniformly, then $\gamma(\cdot, \alpha_n (\cdot)) \to \gamma(\cdot,
\alpha (\cdot)) $ locally uniformly as $\gamma$ is Lipschitz continuous  by \eqref{a:gamma}, and
$\theta_n \to \theta$ locally uniformly, where $\theta_n, \, \theta$ are the
solutions of \eqref{eq:theta1}-\eqref{eq:theta2} corresponding to
$\alpha_n,\,\alpha$, respectively. The reader is referred to
\cite{userguide} or to Section \ref{sec:coupled} below.

As $\Omega(0)$ is a John domain, \cite[Theorem 1.1]{john}
ensures that all sublevels $\Omega(t)  $
are John domains, as well. More precisely,
if  $\Omega(0)$ is a John domain with respect to the point $x_0\in
\Rz^d$ with John constant  $C_0$  then all $\Omega(t) $ are John domains with respect to the same
point $x_0$ and with John constant at least $C_J:=\min\{1, C_0\}
\tfrac{\gamma_*}{(2\gamma^* + \gamma_*)}$.
In particular, one has that $\Omega(t) \in \Theta$  for all $t\in
[0,T]$, where 
\begin{align}
\notag\Theta\coloneqq \big\{ &\Omega  \subset\subset D\ \, : \,
                         \Omega \ \text{is a John domain with
                               respect to the point $x_0\in\omega
                               $,}\\
                         & \qquad \text{with John constant $C_J$, and
                         ${\mathrm{dist}\,(\omega;\partial\Omega)}\geq \rho_0>0$}\}.\label{eq:Theta}
\end{align}
This is a crucial observation, for it entails the validity of a
uniform Korn inequality. Recall that for  any given John
domain 
$\Omega \subset \Rz^d$ there exists a constant $C_{\rm Korn}$ such that
 \begin{equation}
\label{eq:Korn}
 \| \nabla u \|_{L^2(\Omega;\Rz^{d\times d})} \leq C_{\rm Korn}\|
 \epsi(u)\|_{L^2(\Omega;\Rz^{d\times d})} \quad \text{  for all } u \in H^1_0(\Omega;\Rz^d).
 \end{equation}
  In fact, the validity of the Korn inequality is actually equivalent
 to $\Omega$ being John in the special class of domains
 fulfilling the so-called {\it separation property}
 \cite{acosta,jiang}. This include simply connected planar domains \cite{buckley}.

Note that the constant in the Korn
 inequality \eqref{eq:Korn} depends on $\Omega$ only. More precisely,
$C_{\rm Korn}=C_{\rm Korn}( C_{J}, d(x_0, \partial \Omega),{\rm
  diam}(\Omega))$, see \cite[Theorem 4.1]{acosta}.  In the following,
we use the fact that the Korn constant is actually uniform on
$\Theta$.  %for uniformly bounded domains having the same John
               %constant, containing acommon open bounded set, and
               %being John with respect to the same point.
In particular, we have the following.

\begin{proposition}[Uniform Korn inequality]\label{UKI}
Let $\omega, D\subset \Rz^d$ be open bounded domains, with
$x_0\in\omega\subset\subset D$, and $|\omega|>0$.    Define $\Theta$ as in
\eqref{eq:Theta}. %   Let also $x_0\in \Rz^d$. Denote by $\Theta$ the following class of subsets of $\Rz^d$:
% \begin{align*}
% \Theta\coloneqq \big\{ &\Omega:\,\Omega\text{ is an open John domain  with respect to }x_0, \text{ it has John constant }C_{J},\\
% &\text{ and is such that }\omega\subset\bar{\Omega}\subset D\big\},
% \end{align*}
% and set $$H^1_\omega (\Omega;\Rz^d)\coloneqq \{u\in H^1 (\Omega;\Rz^d):\,u=0\text{ on }\omega\}.$$
Then, there exists a constant $ C_{\Theta}=C_{\Theta}(C_J,x_0,\omega, D,\rho_0)$ such that for every $\Omega\in \Theta$ and every $u\in H^1_\omega(\Omega;\Rz^d)$ there holds
\begin{equation}
\label{eq:unif-korn}
\| \nabla u \|_{L^2(\Omega;\Rz^{d\times d})} \leq C_{\Theta}\|
 \epsi(u)\|_{L^2(\Omega;\Rz^{d\times d})}.
\end{equation}
\end{proposition}
 We prove the uniform Korn inequality \eqref{eq:unif-korn} in
Section \ref{sec:korn}. This inequality is paramount for
studying the 
variational equation \eqref{eq:eq1}. By assuming  to be given $t\in [0,T]
\mapsto \Omega(t) \in \Theta$,  as well
as $\alpha\in L^2(D; \Rz^{d\times d})$, one can uniquely
solve \eqref{eq:eq1} for all times $t\in [0,T]$ by means of the
standard Lax-Milgram Lemma, as the coercivity of the corresponding bilinear
form follows from \eqref{eq:unif-korn}.

\subsection{Main results}\label{sec:statements}
 Recall that \eqref{a:T}-\eqref{a:f}  from Section \ref{sec:assumptions} are
assumed throughout. 
% \texttt{E: I think that the set $Q$ and each subset of it are actually always measurable. I am putting a proof in the following lemma. It is taken from: \\
%\href{https://math.stackexchange.com/questions/1565831/measurability-of-e-x-y-fxy}%{https://math.stackexchange.com/questions/1565831/measurability-of-e-x-y-fxy}}\\ 

We are now in the position to state our main results. 

\begin{theorem}[Equilibrium, given the growth]\label{prop:equi}
  Let $\theta : \Rz^d \to [0,\infty) $ be given, so that the
  corresponding set-valued map $t\in [0,T] \mapsto \Omega(t) := \{
  x \in \Rz^d \ : \ \theta(x)<t\} $
  takes values in $ \Theta$. Setting $Q:=\cup_{t\in
    (0,T)} \Omega(t) \times \{t\}$,
  % $\theta: \Rz^d \to [0,\infty)$ be
  % Lipschitz continuous with $1/\gamma^*\leq |\nabla \theta|\leq 1/\gamma_* $ almost everywhere  and such that
  % $\theta(x)=0$ iff $x \in \Omega(0)$. For all $t \in (0,T)$ let $\Omega(t):= \{
  % x \in \Rz^d \, : \, \theta(x)<t\} $ be connected and satisfy
  % the interior ball condition for $\rho>0$.
  there exists a
  unique  measurable function $u: Q\to \Rz^d$ with $ 
u(\cdot,t)\in H^1_{ \omega}(\Omega(t);\Rz^d)$ for almost every $t\in (0,T)$ and
$t \mapsto \|u(\cdot,t)\|_{H^1} \in L^\infty(0,T)$ solving the equilibrium system
\eqref{eq:eq1}-\eqref{eq:eq2}.
\end{theorem}

\begin{theorem}[Coupled equilibrium and growth]\label{thm:coupled}
   For $T>0$ small enough \EEE there exist  a Lipschitz
  continuous \EEE $\theta: \Rz^d
  \to [0,\infty)$ and  a measurable \EEE $u: Q\to
  \Rz^d$, with $ 
u(\cdot,t)\in H^1(\Omega(t);\Rz^d)$ for $\Omega(t):= \{
  x \in \Rz^d \ : \ \theta(x)<t\}$, $Q:=\cup_{t\in(0,T)} \Omega(t)
  \times \{t\}$ for almost every $t\in (0,T)$, solving the coupled equilibrium and
growth problem \eqref{eq:eq1}-\eqref{eq:theta2}.  
\end{theorem}

 Theorems \ref{prop:equi}  and \ref{thm:coupled} are proved in Sections
\ref{sec:equi}  and  \ref{sec:coupled}, respectively.  The 
choice for $\alpha$ in \eqref{eq:eq2}, which is inspired by \cite{zurlo} and
assumes that material is added in a locally unstressed state, can be
generalized. One can assume that material is added to the boundary of
the solid at  point $x$ with some given and possibly nonvanishing 
stress $\haz \sigma (x)$, by prescribing 
$\haz \sigma \in
L^2(\Rz^d;\Rz^{d\times d}_{\rm sym})$. To this end, one would just need to
reformulate position \eqref{eq:eq2} as $\alpha(x) = K\ove
\epsi(u)(x,\theta(x)) - \CC^{-1}\haz \sigma(x)$, where $\CC^{-1}$ is
the {\it compliance} tensor. The above results would still hold under this
generalization.

Before moving on, let us explicitly remark that the smallness
assumption on $T$ in Theorem \ref{thm:coupled}  is not due to the
possible onset of singularities in $t\mapsto \Omega(t)$, which are
 here allowed.  The need for restricting to small times resides in the
very nature of the subsystem \eqref{eq:quasi}--\eqref{eq:alpha}, where
the backstress $ \mathbb C  \alpha$ acts as actual forcing. In absence of
mollification, namely, for $K = {\rm id}$, system
\eqref{eq:quasi}--\eqref{eq:alpha} would correspond to the quasistatic
equilibrium system, with the extra forcing  $\nabla \cdot \mathbb C \epsi
(u)(x,\theta(x))$. This extra forcing depends on a space-time trace of
the solution itself, having the same size of $\mathbb C \epsi(u)$. This represents a clear bottleneck to
compactness, hence to existence. On the contrary, the action of the
compactifying operator $K$ allows {\it for small times} to control the
size of $ \mathbb C \alpha$, so that the extra forcing term $\nabla \cdot \mathbb C \alpha$
is dominated by $\mathbb C \epsi(u)$ and 
can be handled as a perturbation. 

\section{The equilibrium problem for a given growth: Proof of
  Theorem \ref{prop:equi}}\label{sec:equi}

As  $\Omega (t) \subset D$ for all $t\in [0,T]$ we have that %   $1/\gamma^*\leq |\nabla \theta|$ almost everywhere, we
% readily get that $\Omega(T) \subset \Omega(0) +
% B_{T\gamma^*}(0)$. In particular, the Lebesgue $d$-measure $
% |\Omega(T)|$ of $\Omega(T)$ can be bounded in terms of data. More
% precisely, we get 
\begin{equation}
  \label{eq:Omega}
  |\Omega(T)| \leq   |D|.%\left(\frac12 {\text{\rm diam} (\Omega(0))}  +\frac{T}{\gamma_*}\right)^d |B_1(0)|=: A.
\end{equation}

In order to find a unique solution $u$ to
\eqref{eq:eq1}-\eqref{eq:eq2}, we start by arguing locally in time,
looking for a solution on $(0,T_0)$ for some
$T_0 \in (0,T]$ small. Indeed, for  $T_0$ small enough  the well-posedness of system \eqref{eq:eq1}-\eqref{eq:eq2} follows by a contraction
argument in the
 function space 
\begin{align*}
U(T_0) :=&\Big\{u : \displaystyle \cup_{t\in (0,T_0)} \Omega (t) \times \{t\}\to
           \Rz^d \ \text{measurable such that}  \\
&\quad \quad \quad  
u(\cdot,t)\in H^1_{ \omega }(\Omega(t);\Rz^d) \ \text{for a.e.} \ t\in (0,T_0) \\
&\quad \quad \quad
\text{and} \ t \mapsto \|u(\cdot,t)\|_{H^1 (\Omega(t);\Rz^d) } \in L^\infty(0,T_0)\Big\}.
\end{align*}
Note that $U(T_0)$ is a Banach space when endowed with the norm 
$$
\|u \|_{U(T_0)}:= 
\esssup_{t\in (0,T_0)}\|u(\cdot,t)\|_{H^1(\Omega(t);\Rz^d)}.
$$ 
We  additionally introduce the
notation $E(T_0):= \{\epsi(u) \, : \, u \in U(T_0)\}$ for the
corresponding  Banach space of symmetric
gradients and let
$$\| \epsi(u) \|_{E(T_0)} :=
\esssup_{t\in (0,T_0)} \| \epsi(u)(\cdot,t)\|_{L^2(\Omega(t);\Rz^{d\times
    d})} = 
\esssup_{t\in (0,T_0)} \| \ove \epsi(u)(\cdot,t)\|_{L^2(D;\Rz^{d\times
    d})}. $$
By combining \cite[Lemma~3.1]{bojarski} and \cite[Theorem~5.1]{bojarski} (see also the remark right before
\cite[Formula~(5.3)]{bojarski}) we obtain existence of a constant $C_P>0$ such
that the
uniform Poincar\'e inequality
\begin{equation}
  \| v \|_{L^2(\Omega;\Rz^d)} \leq C_P \| \nabla v
\|_{L^2(\Omega;\Rz^{d\times d})} \quad \text{for all } \Omega \in \Theta \
\text{and all } v \in H^1_\omega(\Omega;\Rz^d)\label{eq:bojarski}
\end{equation}
holds. 
%Applying \cite[Lemma 3.1 and Theorem 5.1]{bojarski} (with $l=1$, $p=2$ for the case $d>2$, and with $l=1$, $p=1$ for the case $d=2$) together with Hölder's inequality we obtain for 
An application to $u\in  U(T_0)$ gives 
\begin{equation}
\begin{split}
\label{eq:uKorn1}
\|u \|_{U(T_0)}^2
&=  \esssup_{t\in (0,T_0)} \Big(
\|u(\cdot,t)\|_{L^2(\Omega(t);\Rz^d)}^2 + \|\nabla
u(\cdot,t)\|_{L^2(\Omega(t);\Rz^{d\times d})}^2\Big)\\
& \leq  (C_{P}^2+1)  \esssup_{t\in (0,T_0)} \Big( \|\nabla u(\cdot,t)\|^2_{L^2(\Omega(t);\Rz^{d\times d})}\Big).
\end{split}
\end{equation}
The uniform Korn inequality
\eqref{eq:unif-korn} then gives  
\begin{equation}
\begin{split}
\label{eq:uKorn2}
 \|u \|_{U(T_0)}
  \leq  \haz C_{\Theta} \|
\epsi(u)\|_{E(T_0)}
\leq  \haz C_{\Theta} \| u \|_{U(T_0)}
\end{split}
\end{equation}
 for $\haz C_\Theta :=  (C_P^2+1)^{1/2} C_{\Theta}$. 

%Old:
%$$\|u \|_{U(T_0)} \leq
%C_{\Theta} \|
%\epsi(u)\|_{E(T_0)} \leq  C_{\Theta} \| u \|_{U(T_0)} \quad
%\text{ for all } u\in  U(T_0). $$

Fix now $\tilde u\in U(T_0)$.  As $\epsi(\tilde u ) \in
E(T_0)$  we have that $\ove \epsi(\tilde u) \in L^2( \Rz^d ;\Rz^{d\times
d})$ and we 
can define 
\begin{equation}\label{alpha2}
  \alpha(x):=K\ove \epsi(\tilde
  u)(x,\theta(x)) \quad \text{for all }  x \in \Rz^d.
\end{equation}
We readily check that
$\alpha\in L^2(\Omega(T_0);\Rz^{d\times d})$. Indeed, we have by
assumption \eqref{a:K} together with Young's inequality on
convolutions that  
\begin{align}
\|\alpha\|_{L^2(\Omega(T_0);\Rz^{d\times d})} &\leq |\Omega(T_0)|^{1/2} \| K\ove
  \epsi(\tilde u)\|_{L^\infty( D \times (0,T_0); \Rz^{d\times d})}\nonumber\\
& \hspace{-0.13cm} \stackrel{\eqref{eq:Omega}}{\leq}  |D|^{1/2} \EEE \| k
\|_{L^{ 1}(0,T_0)}\| \phi\|_{L^{ 2}(\Rz^d)} \|  \epsi(\tilde u)\|_{E(T_0)}.\label{alphabound}
\end{align}
Correspondingly,  owing again to the uniform Korn inequality
\eqref{eq:unif-korn}, for all fixed $t \in (0,T_0)$ one finds  by the Lax-Milgram Lemma  a unique
$u(\cdot,t)\in H^1_\omega (\Omega(t);\Rz^d)$ solving the variational
equation \eqref{eq:eq1} with $\alpha$ given by
\eqref{alpha2}. By choosing
$v=u(\cdot,t)$ in equation \eqref{eq:eq1},  and using assumptions \eqref{a:TC} and \eqref{a:f}, we easily 
check that 
\begin{align*}
c_* \| \epsi(u)\|^2_{E(T_0)} &\leq \|\CC\| \,\|
  \alpha\|_{L^2(\Omega(T_0);\Rz^{d\times d})}\|\epsi(u)\|_{E(T_0)} + \|
  f\|_{ C^0([0,T]; \EEE L^2(\Rz^d;\Rz^d))} \| u \|_{U(T_0)}\\[0.1cm]
&\leq \Big( \|\CC\| \,\|
  \alpha\|_{L^2(\Omega(T_0);\Rz^{d\times d})} + 
                                                                       \haz
                                                                       C_{
                                                                       \Theta}
                                                                       \|
  f\|_{ C^0([0,T]; L^2(\Rz^d;\Rz^d))}\Big) \|\epsi(u)\|_{E(T_0)},
\end{align*}
where we denoted by $\| \CC \|$ the Frobenius norm of
the elasticity tensor $\CC$ and where we again used  the uniform Korn inequality
\eqref{eq:unif-korn} and the uniform Poincar\'e inequality
\eqref{eq:bojarski}. This in particular ensures that $\epsi(u)\in {E(T_0)}$,
hence $u\in U(T_0)$, again by the inequality
\eqref{eq:uKorn2}.  % the uniform Korn inequality and \cite{bojarski}. \EEE

We now show that the mapping $S(T_0): U(T_0) \to U(T_0)$ given by
$S(T_0)(\tilde u) = u$ is a contraction for $T_0$  small. To this aim, fix
$\tilde u_1,\, \tilde u_2 \in U(T_0)$, let $\alpha_1(x) = K\ove
\epsi (\tilde u_1)(x,\theta(x))$ and  $\alpha_2(x) = K\ove
\epsi (\tilde u_2)(x,\theta(x))$ for $x\in \Omega(T_0)$, and define $u_1=S(T_0)(\tilde u_1)$, and
$u_2=S(T_0)(\tilde u_2)$. Testing equations \eqref{eq:eq1} written for $u_1$
and $u_2$ by $v=u_1 - u_2$ and taking their difference we deduce   by
assumption \eqref{a:TC} and by following the same arguments as in
estimate \eqref{alphabound}   that 
\begin{align}
  c_* \| \epsi (u_1 - u_2)\|_{E(T_0)} &\leq \| \CC\|\,\| \alpha_1 - \alpha_2
    \|_{L^2(\Omega(T_0);\Rz^{d\times d})}\nonumber\\[1mm]
  & %\hspace{-0.13cm} \stackrel{\eqref{alphabound}}{\leq}  
  \leq \| \CC\|\,   |D|^{1/2} \| K\ove
    \epsi(\tilde u _1 -\tilde u_2)  
    \|_{L^\infty( D \times (0,T_0)\EEE ;\Rz^{d\times d})}  \nonumber\\[1mm]
  &
 \leq \| \CC\| \,  |D|^{1/2} \EEE \| k
\|_{L^{ 1}(0,T_0)}\| \phi\|_{L^{ 2} (\Rz^d)} \| \epsi(\tilde u_1 - \tilde
  u_2)\|_{E(T_0)}.\label{contr}
\end{align}
By using again inequality
\eqref{eq:uKorn2} we then conclude that 
\begin{align*}
  \| u_1 -u_2 \|_{U(T_0)} \leq \frac{\haz C_{ \Theta} \|\CC\|}{c_*} |D|^{1/2} \| k
\|_{L^{ 1}(0,T_0)}\| \phi\|_{L^{ 2} (\Rz^d)} \|  \tilde u_1 - \tilde
  u_2\|_{U(T_0)}.
\end{align*}

Let now $T_0=T/n$, $n\in \Nz$, be so small that  
\begin{equation}
	\frac{ \haz C_{ \Theta} \|\CC\|}{c_*}  |D|^{1/2} \| k
\|_{L^{ 1 }((j-1)T_0,jT_0)}\| \phi\|_{L^{ 2 }(\Rz^d)} < 1 \quad \text{for all } j=1,\dots,n.
\label{condT0}
\end{equation}
Note that such $n$ exists as $k$ is  integrable, see \eqref{a:K}.
Under condition
\eqref{condT0}, 
the mapping $S(T_0)$ is a contraction, hence admitting a unique fixed
point. This proves Theorem \ref{prop:equi} for small times.

We next show that one can obtain a  global solution on $(0,T)$ by
successively solving on $(0,T_0)$, $(0,2T_0)$, \dots, $(0,jT_0)$, for
$j = 1,\dots,n$. Assume to have uniquely solved system  
\eqref{eq:eq1}-\eqref{eq:eq2} on
$(0,(j{-}1)T_0)$. Indicate by $u_j \in U((j{-}1)T_0)$ the corresponding
solution and fix
$$\tilde u \in V:=\{ v \in U(jT_0)\,: \, v = u_j \
\text{on} \ (0,(j{-}1)T_0)\}$$
 which is a closed subspace of $U(jT_0)$. \EEE

By defining $\alpha$ as in \eqref{alpha2}
with $\Omega(T_0)$ replaced by $\Omega(jT_0)$  we can
  reproduce bound \eqref{alphabound}  with $jT_0$ instead of $T_0$
  \EEE so that  $\alpha \in
  L^2(\Omega(jT_0);\Rz^{d \times d})$. One hence finds a unique
  solution $u = S(jT_0)(\tilde u)$  of the equilibrium system \eqref{eq:eq1} with $\alpha$ given by
  \eqref{alpha2}, for almost all $t \in (0,jT_0)$.
  Recall that we have $u
= u_j$ on $(0,(j{-}1)T_0)$, since $u$ is the unique solution to \eqref{eq:eq1} with $\alpha(x)= K\ove \epsi( 
    u_j)(x,\theta(x))$ for $x \in \Omega((j{-}1)T_0)$. In particular,
    $u \in V$. 

We conclude by checking that $S(jT_0) : V \to V$ defined by $\tilde u \mapsto u$ is a
contraction. Fix
$\tilde u_1,\, \tilde u_2 \in V$, let the corresponding
$\alpha_1$ and $\alpha_2$ be given by \eqref{alpha2}, and define $u_1=S(jT_0)(\tilde u_1)$, and
$u_2=S(jT_0)(\tilde u_2)$. We adapt the argument  of estimate
\eqref{contr}, taking into account that $\tilde u_1 = \tilde u_2$ on
$(0,(j{-}1)T_0)$. We  get
\begin{align}
  &c_* \| \epsi (u_1 - u_2)\|_{E(jT_0)}   \leq \| \CC\|\,\| \alpha_1 - \alpha_2
    \|_{L^2( D \EEE ;\Rz^{d\times d})} \nonumber\\[1.5mm]
  &\quad \leq \| \CC\|\,  |D|^{1/2} \EEE \|
    K\ove \epsi(\tilde u_1 - \tilde u_2)\|_{L^\infty( D \times
    \EEE (0,jT_0)  ;\Rz^{d\times d})}\nonumber\\[1.5mm]
  &\quad = \| \CC\|\,  |D|^{1/2} \EEE  \sup_{(x,t)\in  D \EEE \times (0,jT_0)} \left|
    \int_0^t \!\int_{\Rz^d}
    k(t -s) \phi(x-y)\ove\epsi (\tilde u_1(y ,s) - \tilde u_2(y, s))
    \, \d y\,\d s \right|\nonumber\\
 & \quad \leq \| \CC\|\,  |D|^{1/2} \EEE  \sup_{(x,t)\in  D \EEE \times (0,jT_0)} 
    \int_{(j{-}1)T_0}^{t} \!\int_{\Rz^d}
    |k(t -s)|\,| \phi(x-y)| \, |\ove\epsi (\tilde u_1(y ,s) - \tilde u_2(y, s))|
   \, \d y\,\d s \nonumber\\[0.5mm]
  &\quad  \leq \| \CC\|\,  |D|^{1/2} \EEE  \| k \|_{L^{ 1 }((j{-}1)T_0,jT_0)}\|
    \phi\|_{L^{ 2 }(\Rz^d)} \| \epsi (\tilde u_1- \tilde u_2)\|_{E(jT_0)}.\nonumber
\end{align}
Using again  inequality \eqref{eq:uKorn2} we infer that 
$$ 
  \| u_1 -u_2 \|_{U(jT_0)} \leq \frac{ \haz C_{ \Theta} \|\CC\|}{c_*}\,
   |D|^{1/2}  \| k \|_{L^{ 1 }((j{-}1)T_0,jT_0)}\|
  \phi\|_{L^{ 2 }(\Rz^d)} \|  \tilde u_1 - \tilde
  u_2\|_{U(jT_0)}  
  $$
  so that the smallness assumption \eqref{condT0} entails that
  $S(jT_0): V \to V$ is a contraction. This proves the existence of a unique
  solution of problem \eqref{eq:eq1}-\eqref{eq:eq2} almost everywhere on
  $(0,jT_0)$. The assertion follows by letting $j=n$.

   Note that, in order to prove Theorem \ref{prop:equi} one does
  not need to assume the differentiability of the kernels $k$ and
  $\phi $ as in \eqref{a:K} but the weaker requirements  $k \in L^1(0,T)$ and $\phi\in
  L^2(\Rz^d)$ are indeed sufficient. %enough to conclude. \EEE

\section{The coupled equilibrium-growth problem: Proof of Theorem
  \ref{thm:coupled}}\label{sec:coupled}

 The existence of a solution to the equilibrium-growth coupled
system \eqref{eq:eq1}-\eqref{eq:theta2} for $T>0$ small follows by a
fixed-point argument on the function $\alpha$.
Define
$$A:=\{\alpha \in C^0(D;\Rz^{d\times d}) \,  : \, \|
\alpha\|_{W^{1,\infty}(D;\Rz^{d\times d})}\leq L\} $$
where $L>0$ depends just on the data and is specified in \eqref{eq:L} below.

Given $\tilde \alpha \in A$ one has  by \eqref{a:gamma}  that $x\mapsto
\gamma(x,\tilde \alpha(x))$ is Lipschitz continuous and,  as discussed in Subsection \ref{sec:assumptions}, there exists a
unique $\theta $ solving
\eqref{eq:theta1}-\eqref{eq:theta2} with $\gamma(x,\alpha(x))$
replaced by $\gamma(x,\tilde \alpha(x))$. With such $\theta$ one defines $t\mapsto \Omega(t)
=\{x\in \Rz^d \, : \,  \theta(x) <t \}\in \Theta$. As   $t
\mapsto \Omega(t)$ is increasing by set inclusion and $\cup_{t \in
  (0,T)} \Omega(t) \times \{t\} = \{ (x,t) \in \Rz^d \times (0,T)\, :
\, \theta(x)<t\}$ is measurable by Lemma \ref{lemma:set-meas}, one uses Theorem
\ref{prop:equi} in order to find the unique solution $u$ of
\eqref{eq:eq1}-\eqref{eq:eq2} for the given $t\in [0,T]\mapsto \Omega(t)$. This in particular defines the mapping
$$S : \tilde \alpha \in A \subset  C^0(D;\Rz^{d\times d})  \mapsto
\alpha=K\ove \epsi (u) (\cdot, \theta(\cdot)) \in  C^0(D;\Rz^{d\times d}).$$
The assertion of Theorem \ref{thm:coupled} follows as soon as we prove 
that   $S$ admits a fixed point.

To start with, let us provide an a-priori estimate on $u(\cdot,t)$. By
choosing $v =u(\cdot,t)$ in \eqref{eq:eq1} and using inequality
\eqref{eq:uKorn2} 
we get
\begin{align}
   c_*\|  \epsi(u)(\cdot,t)\|_{L^2(\Omega(t);\Rz^{d\times d})}\leq \| \CC \| \,\|
  \alpha\|_{L^2(D;\Rz^{d\times d})} + \haz C_\Theta \| f(\cdot, t) \|_{L^2( \Rz^d;\Rz^{d})}. \label{4:1}
\end{align}
On the other hand,  using \eqref{a:K} and applying Young's inequality
on convolutions we can control $\alpha$ as
\begin{equation}
  \label{4:2}
  \| \alpha\|_{L^\infty(D;\Rz^{d\times d})} \leq \| k \|_{L^1(0,T) } \|
    \phi\|_{L^2(\Rz^d)} \|  \epsi (u) \|_{E(T)}.
  \end{equation}
  We now assume that $T>0$ is so small that
  \begin{equation} \frac{\| \CC \| \,|D|^{1/2}}{c_*}\| k \|_{L^1(0,T) } \|
  \phi\|_{L^2(\Rz^d)}  =: \eta <1\label{eta}
  \end{equation}
  and combine \eqref{4:1}-\eqref{4:2} in order to get that
  $$ c_*(1-\eta)  \| \epsi (u) \|_{E(T)}  \leq  \haz
  C_\Theta \|
  f  \|_{C^0([0,T];L^2( \Rz^d;\Rz^{d}))}.$$
    This in particular entails that
    \begin{equation}
      \label{4:3}
 \| \epsi (u) \|_{E(T)}  \leq
\frac{\haz C_\Theta }{ c_*(1-\eta)}\| f \|_{C^0([0,T];L^2(\Rz^d;\Rz^d))}=:M
\end{equation}
where $M$ depends on data only.  
Note that the smallness assumption
\eqref{eta} does not require the smallness of applied forces.

We now compute the gradient  
$$ \nabla \alpha(x) = (\nabla K \ove \epsi (u)) (x, \theta(x)) +
(\partial_t  K \ove \epsi (u)) (x, \theta(x)) \, \nabla \theta(x)
\qquad \text{for all } x \in D.$$
Using \eqref{4:2}-\eqref{4:3}, the regularity of the kernels $k$ and $\phi$
from \eqref{a:K},  and Young’s inequality on convolutions we hence have that  
\begin{align}
  \| \alpha\|_{L^\infty(D;\Rz^{d\times d})}  &\leq \| k \|_{L^1(0,T) } \|
  \phi\|_{L^2(\Rz^d)} M, \label{4:4}\\[0.1cm]
  \| \nabla \alpha\|_{L^\infty(D;\Rz^{d\times d\times d})}   &\leq \| k \|_{L^1(0,T) } \|
   \nabla \phi\|_{L^2(\Rz^d)} M \nonumber\\
  & \quad  + \Big(\| k'\|_{L^1(0,T) } + |k(0)| \Big)\|
  \phi\|_{L^2(\Rz^d)} \| \nabla \theta \|_{L^\infty(D;\Rz^d)}M. \label{4:5}
\end{align}
Thus, recalling \eqref{eq:gradtheta} and  letting 
\begin{equation}
  \label{eq:L}
 L:=   \|k
    \|_{L^1(0,T)} \| \phi\|_{H^1(\Rz^d)}M + \Big(\|k'
    \|_{L^1(0,T)}+ |k(0)|\Big) \|\phi\|_{L^2(\Rz^d)} \frac{1}{\gamma_*}
   M 
 \end{equation}
 one has that $\alpha = S(\tilde \alpha)$ belongs to $A$, as well. Note
that $L$ is bounded in terms of data only.

We now check the continuity of $S$ with respect to the strong topology 
of  $C^0(D;\Rz^{d\times d})$. Let $\tilde \alpha_n,
\tilde \alpha \in C^0(D;\Rz^{d\times d})$ be given with $\tilde \alpha_n
\to \tilde \alpha $ uniformly. As $\gamma $ is Lipschitz continuous, see 
\eqref{a:gamma}, we have that $\gamma(\cdot, \tilde \alpha_n(\cdot))
\to \gamma(\cdot, \tilde \alpha(\cdot))$ uniformly, as well. This
suffices to pass to the limit in the eikonal problem
\eqref{eq:theta1}-\eqref{eq:theta2} written for $\gamma(\cdot, \tilde
\alpha_n(\cdot))$ and to find that the corresponding solutions
$\theta_n$ converge uniformly to the solution $\theta$ of
\eqref{eq:theta1}-\eqref{eq:theta2}   for $\gamma(\cdot, \tilde
\alpha(\cdot))$. Indeed, the functions
$\theta_n$ are uniformly Lipschitz continuous with $\theta_n=0$ on
$\overline{\Omega(0)}$. They hence admit a not relabeled, locally
uniformly converging subsequence $\theta_n \to \theta$  with $\theta$
Lipschitz continuous and $\theta=0$ on
$\overline{\Omega(0)}$. Let $x_0\in \Rz^d \setminus
\overline{\Omega(0)}$ be given and $\varphi$ be smooth with
$\varphi(x_0)=-\theta(x_0)$ and $\varphi \geq - \theta$ in a
neighborhood of $x_0$. By using the classical approximation procedure
of \cite[Proposition~2.4]{userguide}, we find $x_n \in \Rz^d \setminus
\overline{\Omega(0)}$ with $x_n \to x_0$ and
$\varphi_n$ smooth such that 
$\varphi_n(x_n)=-\theta_n(x_n)$ and $\varphi_n \geq - \theta_n$ in a
neighborhood of $x_n$, and $\nabla \varphi_n(x_n) \to \nabla
\varphi(x_0)$. As $\theta_n$ are viscosity subsolutions, we have that
$ \gamma(x_n,\alpha_n(x_n)) |\nabla
\varphi_n(x_n)| \leq 1$. By passing to the limit as $n\to \infty$ we
obtain that  $\gamma (x_0,\alpha(x_0))|\nabla \varphi(x_0)|\leq 1$,
so that $\theta$ is a viscosity subsolution, as well. In a similar way, we
can check that $\theta$ is a viscosity supersolution, hence a viscosity
solution. Given uniqueness, no extraction of subsequences was actually
needed at this point.

Given such $\theta_n $ and $\theta$ we can define the
corresponding $t\mapsto \Omega_n(t)$ and $t\mapsto \Omega (t)$ (both
increasing by set inclusion and such that the corresponding $Q_n =
\cup_{t\in [0,T]}\Omega_n(t)\times\{t\}$  and
$Q =
\cup_{t\in [0,T]}\Omega(t)\times\{t\}$ are measurable). As $\theta_n$ converges to $\theta$ locally uniformly and the
inequalities \eqref{eq:gradtheta} hold, independently of $n$, we have
that $\Omega_n(t)\to \Omega(t)$ in the Hausdorff sense, uniformly with
respect to $t \in [0,T]$. Moreover, $Q_n\to Q$  in the Hausdorff sense, as
well.

Let us 
indicate by $u_n$ and $u$ the unique solutions of \eqref{eq:eq1}-\eqref{eq:eq2} given by Theorem
\ref{prop:equi}. From the very definition of $S$ let us recall that $\alpha_n
=S(\tilde \alpha_n) = K\ove\epsi(u_n)(\cdot,\theta_n(\cdot))$ and $\alpha
=S( \tilde \alpha ) = K\ove\epsi(u)(\cdot,\theta(\cdot))$.
Bounds \eqref{4:3}-\eqref{4:5}  and a localization argument entail that
\begin{align}
  \ove \epsi(u_n) \rightharpoonup^\ast \ove \epsi(u) \quad&\text{weakly* in} \ \  L^\infty(0,T;L^2(D; \Rz^{d\times  d})).\label{4:6} 
\end{align}
In fact, one has that $\ove \epsi(u_n)$ are uniformly bounded in
$L^\infty(0,T;L^2(D; \Rz^{d\times  d}))$, hence admitting a weak$*$
limit along some not relabeled subsequence. Denote by $\tilde \epsi$
such limit. Fix now $(x,t)\in Q$, as well as $\eta>0$ small enough, so
that $Q_\eta = B_\eta(x) \times (t-\eta,t+\eta) \subset \subset Q$. From the convergence $Q_n \to Q$ in the Hausdorff sense we have
that $ Q_\eta\subset Q_n$ for all $n$ large enough. Hence, by
indicating by $1_{ Q_\eta}$ the indicator function of $Q_\eta$ one has
that
$\ove \epsi(u_n) 1_{Q\eta} \to \tilde \epsi 1_{Q\eta}$ weakly$*$ in
$L^\infty(t-\eta,t+\eta,L^2(B_\eta(x); \Rz^{d\times d}))$. At the same
time $\ove \epsi(u_n) 1_{Q\eta}  = \epsi (u_n) 1_{Q\eta} \to  \epsi(u) 1_{Q\eta}$ weakly$*$ in
$L^\infty(t-\eta,t+\eta,L^2(B_\eta(x); \Rz^{d\times d}))$. As $(x,t)
\in Q$ is arbitrary, this shows that $\tilde \epsi = \epsi(u) \equiv
\ove \epsi(u)$ on
$Q$. An analogous argument applied to $(x,t) \not \in \ove Q$ proves
that $\tilde \epsi=0=\ove \epsi(u)$ in $\Rz^d \times (0,T) \setminus
\ove Q$.
In order to conclude for \eqref{4:6} it hence suffices to recall
that $\partial Q$ is negligible. 

Note that the whole sequence $\ove \epsi(u_n)$ converges, due to the uniqueness of the
limit. Owing to the compactifying character of the nonlocal operator $K$ we
also have that
\begin{align}
  K\ove \epsi(u_n) \to K \ove \epsi(u)\quad &\text{strongly in} \ \
                                              C^0(D\times (0,T); \Rz^{d\times d}).\label{4:8}
\end{align}
In addition, $K\ove \epsi(u_n)$ are uniformly Lipschitz
continuous in time: By following the argument leading to bounds
\eqref{4:4}-\eqref{4:5}, we can check that 
\begin{align}
 \| K\ove \epsi(u_n)(\cdot,t_1) -K\ove \epsi(u_n)(\cdot  ,t_2)
  \|_{L^\infty(D;\Rz^{d\times d})}  & \leq \|
  \partial_t   K\ove \epsi(u_n) \|_{L^\infty(D\times (t_1,t_2);\Rz^{d\times
  d})} |t_1 - t_2| 
  \nonumber \\ 
  &  \leq \big( \| k'\|_{L^1(t_1,t_2)} + |k(0)| \big)
    \|\phi\|_{L^2(\Rz^d)}   M\, |t_1 - t_2| \label{4:9}
\end{align}
for all $0<t_1 <t_2<T$. 
We can hence conclude that
\begin{align}
  &\| \alpha_n - \alpha \|_{L^\infty(D;\Rz^{d\times d})}  \nonumber\\[2mm]
  & \quad = \| K\ove \epsi(u_n) (\cdot,
  \theta_n(\cdot)) - K\ove \epsi(u) (\cdot,
  \theta(\cdot))\|_{L^\infty(D;\Rz^{d\times
    d})}    \nonumber\\[2mm]
  & \quad \leq \| K\ove \epsi(u_n) (\cdot,
  \theta_n(\cdot)) - K\ove \epsi(u_n) (\cdot,
\theta(\cdot))\|_{L^\infty(D;\Rz^{d\times
    d})}  + \| K\ove \epsi(u_n) (\cdot,
  \theta(\cdot)) - K\ove \epsi(u) (\cdot,
    \theta(\cdot))\|_{L^\infty(D;\Rz^{d\times d})} \nonumber\\
  &\; \stackrel{\eqref{4:9}}{\leq}\big( \| k'\|_{L^1(0,T)} + |k(0)|
  \big) \|\phi\|_{L^2(\Rz^d)} \EEE M \| \theta_n - \theta\|_{L^\infty(D)} + \| K\ove \epsi(u_n) - K\ove \epsi(u) \|_{L^\infty(D \times
    (0,T);\Rz^{d\times d})} \to 0,\nonumber
\end{align}
where we have used that $\theta_n \to \theta $ and
$K\ove \epsi(u_n) \to K\ove \epsi(u)$ uniformly. This proves that   $S(\tilde \alpha_n) \to S(\tilde \alpha)$ strongly in
$C^0(D;\Rz^{d\times d})$, namely, that  $S$ is continuous.

Eventually, as $A$ is   convex and compact in
$C^0(D;\Rz^{d\times d})$ we can apply the Schauder Fixed-Point Theorem
and complete the proof of Theorem \ref{thm:coupled}.

\section{Uniform Korn inequality}
 \label{sec:korn}
 We conclude this paper with a proof of the uniform Korn inequality in
 Proposition \ref{UKI}. 

Let us argue by contradiction and assume that the statement  of Proposition
\ref{UKI} is false.  In particular, for every $k\in
\mathbb{N}$ we assume to be given an open set $\Omega_k\in \Theta$ and a map $u_k\in H^1_\omega( \Omega_k \EEE;\Rz^d)$ such that
$$\| \nabla u_k \|_{L^2(\Omega_k;\Rz^{d\times d})} > k\|
 \epsi(u_k)\|_{L^2(\Omega_k;\Rz^{d\times d})}.$$
 By normalizing, with no loss of generality we can assume that
 \begin{equation}
 \label{eq:normalized}\| \nabla u_k \|_{L^2(\Omega_k;\Rz^{d\times d})}=1\quad\text{and}\quad\|
 \epsi(u_k)\|_{L^2(\Omega_k;\Rz^{d\times d})}\leq\frac1k
 \end{equation}
 for every $k\in \mathbb{N}$. As all $\Omega_k \subset\subset
 D$ and $D$  is bounded, one can find a not relabeled subsequence such
 that $\ove{\Omega}_k\to K $ in the sense of the
 Hausdorff convergence, where $K\subset \subset D$. 
 %Since the Hausdorff convergence induces a complete metric on the setof closed nonempty bounded sets contained in a fixed domain $D$, it follows in particular that the limit of the sequence of sets$\{\Omega_k\}_{k\in\Nz}$ is still a compact set. 
Define now $\Omega_\infty:={\rm int} (K)$ (interior part), so that $\ove{\Omega}_\infty=K$. % is then defined as the interior part of such Hausdorff limit. 
 From the connectedness of each set $\Omega_k$ we also infer that $\ove{\Omega}_\infty$ is connected. 
 %\todo[inline]{
 %E: the idea of the proof of this should be the following. By contradiction, if a compact set is not connected, it has at least two connected components $A$ and $B$ at positive distance, say $d>0$ from each other. Being the Hausdorff limit of a sequence of sets, for $\epsi$ small enough the sets must completely lie within an $\epsi$-neighborhood of $A\cup B$ and there must be both points close to $A$ and close to $B$ so that for $\epsi$ small the sets should become disconnected, leading to a contradiction. It should be something classical/basic. We should think whether including it or not.
% }
 
 Define now  $S_k\coloneqq \cap_{n\geq k} \ove{\Omega}_n$, 
 so that the sequence $\{S_k\}_{k\in \Nz}$ 
 is increasing by set inclusion. In particular, there also holds
 $S_k\cap \ove{\Omega}_\infty\to \ove{\Omega}_\infty$ in the Hausdorff
 sense.  We notice that the Korn constant $\haz C_{\rm Korn}$ in the first Korn inequality
 is the same for all sets in $\Theta$, cf. 
 \cite[Theorem~4.2]{acosta}, as such constant depends on $C_J$, $d$,
 $\rho_0$, and ${\rm diam}(\Omega)$ only,  \cite[Theorem~4.1]{acosta}.  %\cite[Theorems 3.8 and 5.17]{diening-et-al} as well as \cite[Lemma 3.1]{bojarski} and the discussion below. 
 Thus, we infer that
 \begin{align}
 \nonumber
 1&=\| \nabla u_k \|_{L^2(\Omega_k;\Rz^{d\times d})}^2
 \nonumber \\
 &\leq  \haz C_{\rm Korn} \left(\|u_k\|_{L^2(\Omega_k;\Rz^d)}^2+\|\epsi(u_k)\|_{L^2(\Omega_k;\Rz^{d\times d})}^2\right) \nonumber\\
\label{eq:contradiction} &\leq 2 \haz C_{\rm Korn}  \left(\|u_k\|_{L^2(S_k\cap \ove{\Omega}_\infty;\Rz^d)}^2+\|u_k\|_{L^2(\Omega_k\setminus(S_k\cap \ove{\Omega}_\infty);\Rz^d)}^2+\|\epsi(u_k)\|_{L^2(\Omega_k;\Rz^{d\times d})}^2\right).
 \end{align} 
 In view of \eqref{eq:normalized}, we already know that the third term
 on the right-hand side of \eqref{eq:contradiction} converges to 
 $0$ 
 as $k\to \infty$. In order to reach a contradiction, we proceed by
 showing that also the first and second contributions on the
 right-hand side of \eqref{eq:contradiction}  are infinitesimal as
 $k\to \infty$.

 We subdivide the remaining part of the proof into
  four steps. 
 
  \noindent\textbf{Step 1.} We first show that the set $\Omega_\infty$
  is still a John domain with respect to $x_0$, possibly with a
  smaller John constant. 

  Let $x\in \Omega_\infty$ be fixed.  First, recall that
  $\omega$ is connected and since ${\rm dist}\,(\omega;\partial
  \Omega_k)\geq \rho_0>0$ for every $k\in \Nz$, there holds
  $\omega\subset\subset\Omega_\infty$. Therefore, if  $x\in\omega$,
  the existence of an arc-length parametrized curve $ \rho $
  joining $x$ and $x_0$ and with positive distance from
  $\partial\Omega_\infty$ is directly ensured.
  
  Consider now the case in which $x\notin \omega$. For $k$ big enough
  $x\in \Omega_k$, and there exists an arc-length parametrized curve
  $ \rho_k  :[0,L_k]\to \Omega_k$ such that $ \rho_k 
  (0)=x$, $ \rho_k  (L_k)=x_0$ and ${\rm dist}\,(\rho_k(s);\partial \Omega_k)\geq  C_J  s$  for all $s\in [0,
   L_k ]$. The fact that $x\notin \omega$ implies that  $
  L_k \geq {\rm dist}\,(x_0;\partial\omega)$ for all $k\in
  \Nz$. Without introducing new notation, we extend each curve $
  \rho_k $ continuously to the whole interval $[0,\infty)$, by setting $
  \rho_k  (s)=x_0$ for every $s>L_k$.
  
One has that $\sup_{k\in \Nz}\|
  \rho_k \|_{L^\infty(0,\infty)} < \infty$ due to the
  fact that $\Omega_k\subset D$ for every $k\in \Nz$ and to the
  boundedness of $D$. Moreover,  $\sup_{k\in \Nz}\|\dot{
  \rho_k }\|_{L^\infty(0,\infty)} \leq 1$ because all curves $
  \rho_k $ are parametrized by arc-length on $[0,L_k]$ and are
  then constant. As a result, $ \{
  \rho_k \}_{k\in \Nz} \subset W^{1,\infty}(0,\infty)$ is
  uniformly bounded and there exists $L\in (0,\infty]$ and a curve
  $ \rho :[0,\infty)\to D$ with $ \rho  (0)=x$, $
  \rho  (L)=x_0$, such that, up to subsequences $L_k\to L$ and
  $\rho_k \to  \rho $ strongly in
  $L^{\infty}(0,\infty)$ and weakly* in $W^{1,\infty}(0,\infty)$. 

  We proceed by showing that $ \rho(s) \subset \Omega_\infty$
  for all $s\geq 0$  and that ${\rm dist}\,( \rho(s);\partial\Omega_\infty)\geq  { C_J s}/{8}$ for all $s\in
  [0,L]$. Indeed, fix $\ove{s}>0$ and let $\ove{k}\in \Nz$ big enough
  so that 
  $$
  \| \rho_k -\rho\|_{L^{\infty}(0,\infty)}\leq \frac{C_J\ove{s}}{8}\quad\text{for
    every} \ \ k\geq \ove{k}.
    $$
  Then,   
  $$
  B_{\frac{ C_J \ove{s}}{8}}( \rho  (\ove{s}))\subset B_{\frac{ C_J \ove{s}}{8}}( \rho_k  (\ove{s}))+B_{\frac{ C_J \ove{s}}{8}}(0)\subset\subset B_{\frac{ C_J \ove{s}}{4}}( \rho_k  (\ove{s}))\subset\subset \Omega_k
  $$
  for every $k\geq \ove{k}$. In particular, it follows that
  $B_{ C_J \ove{s}/{8}}( \rho
  (\ove{s}))\subset\Omega_\infty$. Since the same argument holds for
  every $\ove{s}>0$, we deduce that $ \rho (0,\infty) \subset
  \Omega_\infty$. From the openness of $\Omega_\infty$ we then infer
  that also $x= \rho  (0)\in \Omega_\infty$, as well, and that $\Omega_\infty$ is a John domain with respect to
  $x_0$, with John constant at least ${C_J}/{8}$.  
  
 \noindent\textbf{Step 2.}  In this step, we prove that the first
 term in the right-hand side of \eqref{eq:contradiction} is
 infinitesimal as $k\to \infty$. 
 
  Since $S_k\subset \Omega_k$ for all $k\in \mathbb{N}$, by \eqref{eq:normalized} it follows that
 $$\|
 \epsi(u_k)\|_{L^2(S_k\cap \Omega_\infty;\Rz^{d\times d})}\leq\frac1k$$
 for all $k\in \mathbb{N}$. Let $\eta>0$. In view of Step 1, by \cite[Theorems 4.5 and 4.6]{vaisala}, for every $\eta>0$ there exists a John domain $\omega_\eta\subset \Rz^d$ such that $\omega\subset\subset \omega_\eta\subset\subset \Omega_\infty$, and 
  \begin{equation}
  \label{eq:meas-omega-eta}
  |\Omega_\infty\setminus\omega_\eta|<\eta.
  \end{equation} 
Then, $\omega_\eta\subset\subset \Omega_k\cap\Omega_\infty$ for $k$ big enough, and hence $\omega_\eta\subset\subset S_k$ for $k$ big enough. Additionally, $\{u_k\}_{k\in \Nz}\subset H^1_{\omega}(\omega_\eta;\Rz^d)$ for $k$ big enough 
and 
 \begin{equation}
 \label{eq:bd-euk}
 \|\epsi(u_k)\|_{L^2(\omega_\eta;\Rz^{d\times d})}\leq\frac1k.
 \end{equation}
 Since $\omega_\eta$ is a John domain, by \eqref{eq:Korn} and in view of \cite[Theorem~5.1]{bojarski} we infer the existence of a map 
 $u\in H^1_{\omega}(\omega_\eta;\Rz^d)$ such that, up to extracting a further non-relabeled subsequence, there holds
  $$
  u_k\rightharpoonup u\quad\text{weakly in }H^1_{\omega}(\omega_\eta;\Rz^d).
  $$
 On the other hand, by \eqref{eq:bd-euk} we find that $\epsi(u)=0$ on $\omega_\eta$. Since $u=0$ on $\omega$, this implies that $u\equiv 0$ on $\omega_\eta$. Hence, in particular,
 \begin{equation}
 \label{eq:need1}
 u_k\to 0\quad\text{strongly in }L^2(\omega_\eta;\Rz^d).
\end{equation}

As $\Omega_\infty$ is a John domain, the boundary $\partial \Omega_\infty$ has zero measure,
cf. \cite[Corollary~2.3]{koskela-rohde}. Hence, we can write   
 \begin{equation}
 \label{eq:need2}
 \|u_k\|^2_{L^2(S_k\cap \ove \Omega_\infty;\Rz^d)}=  \|u_k\|^2_{L^2(S_k\cap \Omega_\infty;\Rz^d)}=\|u_k\|^2_{L^2(\omega_\eta;\Rz^d)}+\|u_k\|^2_{L^2((S_k\cap \Omega_\infty)\setminus \omega_\eta;\Rz^d)},
\end{equation}
and prove that it converges to $0$ as $k\to \infty$. Indeed, the
first term in the above right-hand side is infinitesimal due to
\eqref{eq:need1}. In order to handle the second term, we first apply
the H\"older inequality and then we rely again on
\cite[Theorem~5.1]{bojarski}. Let us momentarily assume that $d>2$ (the case
$d=2$ is discussed afterwards). By letting  $2^*=2d/(d-2)$, one argues as
follows
\begin{align}
  \|u_k\|^2_{L^2((S_k\cap\Omega_\infty)\setminus\omega_\eta;\Rz^d)}&\leq |(S_k\cap
                                              \Omega_\infty)\setminus
                                              \omega_\eta|^{\frac{2^*-2}{2^*}}
                                              \|u_k\|^2_{L^{2^*}((S_k\cap
                                                                     \Omega_\infty)
                                                                     \setminus\omega_\eta;\Rz^d)}\nonumber\\
  & \leq  |\Omega_\infty\setminus \omega_\eta| ^{\frac{2^*-2}{2^*}}\|u_k\|^2_{L^{2^*}((S_k\cap
                                                                     \Omega_\infty)
                                                                     \setminus\omega_\eta;\Rz^d)}\nonumber\\
  &  \leq C |\Omega_\infty\setminus \omega_\eta| ^{\frac{2^*-2}{2^*}} \| \nabla u_k
    \|^2_{L^2(\Omega_k;\Rz^{d\times d})} \leq
    C\eta^{\frac{2^*-2}{2^*}}, \label{eq:hol}
\end{align} 
where $C$ is independent of $k$. 
 From the arbitrariness of $\eta$, the decomposition
 \eqref{eq:need2} and convergences \eqref{eq:need1} and
 \eqref{eq:hol} entail that 
 \begin{equation}
 \label{eq:Ito0}
 \limsup_{k\to \infty} \|u_k\|^2_{L^2(S_k\cap \Omega_\infty;\Rz^d)}=0.
\end{equation}
We reach the same conclusion in case $d=2$. By applying the first
H\"older step in \eqref{eq:hol} with respect to an arbitrary
exponent $p/2>1$ and then argue via \cite[Theorem~5.1]{bojarski} one gets
$$\|u_k\|^2_{L^2((S_k\cap
  \Omega_\infty)\setminus\omega_\eta;\Rz^2)} \leq
C\eta^{\frac{p-2}{p}}$$
so that \eqref{eq:Ito0} again follows. 

% \EEE
%  On the other hand, recalling that the Sobolev-Poincar\'e constant is invariant on $\Theta$ (see \cite[Section 6]{bojarski}) \LLL (Note: rewrite this part using also Theorem 5.1 in \cite{bojarski} for the case $d=2$), 
%  we have
%  \begin{equation}
%  \label{eq:SP}\|u_k\|^2_{L^{2^\ast}(\Omega_k;\Rz^d)}\leq C_{SP}\|\nabla u_k\|_{L^2(\Omega_k;\Rz^{d\times d})}^2.
%  \end{equation}
%  where $2^\ast$.....\EEE
%  Thus, by H\"older inequality we conclude that
%  \begin{equation}
%  \label{eq:need3}
%  \|u_k\|^2_{L^2((S_k\cap \Omega_\infty)\setminus \omega_\eta;\Rz^d)}\leq C_{SP}\|\nabla u_k\|_{L^2(\Omega_k;\Rz^{d\times d})}^\alpha |\Omega_\infty\setminus \omega_\eta|^\beta\leq C|\Omega_\infty\setminus \omega_\eta|^\beta
%  \end{equation} for suitable $\alpha,\beta>0$. Combining \eqref{eq:need1}--\eqref{eq:need3} we infer the bound
% $$ \limsup_{k\to \infty} \|u_k\|^2_{L^2(S_k\cap \Omega_\infty;\Rz^d)}\leq C|\Omega_\infty\setminus \omega_\eta|^\beta\leq C\eta^\beta.$$
%  From the arbitrariness of $\eta$ we deduce
%  \begin{equation}
%  \label{eq:Ito0}
%  \limsup_{k\to \infty} \|u_k\|^2_{L^2(S_k\cap \Omega_\infty;\Rz^d)}=0.
%  \end{equation}
 
  \noindent\textbf{Step 3.} Here we show that the second term in
  the right-hand side of inequality \eqref{eq:contradiction} goes to
  $0$ as $k\to \infty$. More precisely, we show that 
  \begin{equation}
  \label{eq:IIto0}
   \lim_{k\to \infty} \|u_k\|_{L^2(\Omega_k\setminus(S_k\cap \Omega_\infty);\Rz^d)}=0.
   \end{equation}
   Note that $\Omega_k\setminus(S_k\cap \Omega_\infty)\subset U_k$ where
   $$U_k\coloneqq \left(\cup_{n\geq k} \Omega_n\right)\cap (S_k\cap \Omega_\infty)^c.$$
   By definition, $\{U_k\}_{k\in \Nz}$ is the intersection of two
   decreasing sequences of nested sets and it is thus a
   decreasing sequence of nested sets as well. In particular, 
   $U_k\to \cap_{k\in \mathbb{N}} \,U_k\subset \partial \Omega_\infty$
   both in the Hausdorff sense and in measure. As the  Lebesgue measure of the boundary $\partial \Omega_\infty$ is
   $0$, we have that $|U_k| \to 0$ as $k\to \infty$. For $d>2$, by replicating
    the argument leading to \eqref{eq:hol}, now on the sets $\Omega_k\setminus(S_k\cap
     \ove\Omega_\infty)$, 
    one gets 
 $$ 
  \|u_k\|^2_{L^2(\Omega_k\setminus(S_k\cap
     \Omega_\infty);\Rz^d)}\leq C|U_k|^{ \frac{2^*-2}{2}},
 $$
   which implies \eqref{eq:IIto0}. In the case $d=2$, we argue analogously.
   %  Since the
   % Lebesgue measure of \UUU the boundary $\partial \Omega_\infty$ is
   % zero, \EEE
   % cf. \cite[Cor.~2.3]{koskela-rohde},  \EEE Property \eqref{eq:SP}  and H\"older inequality yield
   % $$\|u_k\|^2_{L^2(\Omega_k\setminus(S_k\cap
   %   \Omega_\infty);\Rz^d)}\leq C|U_k|^\beta$$ for a suitable
   %   $\beta>0$, and hence implies \eqref{eq:IIto0}.

   \noindent\textbf{Step 4.} We are now in the position of concluding
   the proof of Proposition \ref{UKI}. Indeed, the convergences  \eqref{eq:normalized},
   \eqref{eq:Ito0}, and \eqref{eq:IIto0} ensure that the right-hand
   side of inequality
   \eqref{eq:contradiction} converges to $0$ as $k\to \infty$, leading
   to a contradiction.  
   % The statement
   % follows then by combining \eqref{eq:contradiction} with
   % \eqref{eq:normalized}, \eqref{eq:IIto0}, and \eqref{eq:Ito0}.
   \qedhere

%%%%%%%%%%%%%%%%%%%%%%%%%%%%%%%%%%%%%%%%%%
\section*{Acknowledgements}
Support from the Austrian Science Fund (FWF) through
projects 10.55776/F65, \linebreak 10.55776/I4354, 10.55776/I5149,
10.55776/P32788, 10.55776/I4052, 10.55776/V662, \linebreak 10.55776/P35359 and
10.55776/Y1292
as well as from BMBWF
through the OeAD-WTZ project CZ 09/2023 is gratefully acknowledged.
% E.D. acknowledges support from the Austrian Science Fund (FWF) through projects F\,65,  I\,4052, V\,662, and Y\,1292, as well as from BMBWF through the OeAD-WTZ project CZ04/2019. K.N. is partially supported by the Austrian Science Fund (FWF) project  F\,65. U.S. is supported by the Austrian Science Fund (FWF)
% projects F\,65, I\,4354, I\,5149, and P\,32788.
 %%%%%%%%%%%%%%%%%%%%%%%%%%%%%%%%%%%%%%%%%%


\begin{thebibliography}{50}

\bibitem{abi}
R.~Abi-Akl, T.~Cohen. Surface growth on a deformable spherical substrate. {\it Mech. Res. Comm.} 103 (2020), 103457. 


\bibitem{acosta}
G.~Acosta, R.~G.~Dur\'an, M.~Muschietti. Solutions of the divergence
operator on John domains. {\it Adv. Math.} 206 (2006), no. 2,
373--401.

\bibitem{Bardi}
M.~Bardi, I.~Capuzzo-Dolcetta. {\it Optimal control and viscosity solutions of Hamilton-Jacobi-Bellman equations}. Systems \& Control: Foundations \& Applications. Birkh\"auser Boston, Inc., Boston, MA, 1997.

\bibitem{Barles}
G.~Barles. {\it  Solutions de viscosit\'e des \'equations de
  Hamilton-Jacobi}.  Mathematics \& Applications, 17. Springer-Verlag, Paris, 1994.
  
  
   \bibitem{barreira-et-at}
  R.~Barreira, C.~M.~Elliott, A.~Madzvamuse. The surface finite element method for pattern formation on evolving biological surface. {\it J. Math. Biol.} 63 (2011), no. 6,
  1095--1119.
  
  

  \bibitem{bergel}
  G.~L.~Bergel, P.~Papadopoulos. A finite element method for modeling surface growth and
  resorption of deformable solids. {\it Comput. Mech.} 68 (2021), no. 4,
  759--774.
 


  \bibitem{bojarski}
   B. Bojarski. Remarks on Sobolev imbedding inequalities, in {\it Lecture Notes in
Math.} 1351, 52--68, Springer-Verlag, Berlin, 1989.

\bibitem{bressan}
A.~Bressan, M.~Lewicka. A model of controlled growth. {\it
  Arch. Ration. Mech. Anal.} 227 (2018), no. 3, 1223--1266.


\bibitem{brown}
C.~B.~Brown, L.~E.~Goodman. Gravitational stresses in accreted bodies.  {\it Proc. R. Soc. Lond. A} 276 (1963), no. 1367, 571--576. 


\bibitem{buckley}
S.~Buckley, P.~Koskela. Sobolev-Poincar\'{e} implies John. {\it
  Math. Res. Lett.} 2 (1995), no. 5, 577--593.
  
\bibitem{userguide}
M.~G.~ Crandall, H.~Ishii, P.-L.~Lions.
User's guide to viscosity solutions of second order partial differential equations.
{\it Bull. Amer. Math. Soc. (N.S.)}, 27 (1992), no. 1, 1--67. 


\bibitem{morpho}
   E.~Davoli, K.~Nik, U.~Stefanelli.
Existence results for a morphoelastic model. {\it  ZAMM
  Z. Angew. Math. Mech.}  103 (2023), no. 7, Paper No. e202100478,
25 pp. 


\bibitem{john}
  E.~Davoli,  U.~Stefanelli.
  Level sets of eikonal functions are John regular. 
  Preprint \href{https://arxiv.org/abs/2312.17635}{arXiv: 2312.17635},
  2023. 


%   \bibitem{diening-et-al}
%   L.~Diening, M.~R\r{u}\v{z}i\v{c}ka, K.~Schumacher. A decomposition
%   technique for John domains. {\it U. Ann. Acad. Sci. Fenn. Math.}
% 35 (2010), 87--114. 



\bibitem{drozdov}
A.~D.~Drozdov. {\it Viscoelastic structures: mechanics of growth and aging.} Academic Press, Cambridge, 1998. 


\bibitem{egberts-et-al}
G.~Egberts, F.~Vermolen, P.~Van Zuijlen. Stability of a one-dimensional morphoelastic model for post-burn contraction. {\it J. Math. Biol.}  83 (2021), no. 3, Paper No. 24, 35 pp.  





\bibitem{epstein}
M.~Epstein. Kinetics of boundary growth. {\it Mech. Res. Comm.} 37 (2010), no.5, 453--457.


%\LLL
%\bibitem{erlich-et-al}
%A.~Erlich, D. E.~Moulton, A.~Goriely. Are homeostatic states stable? {D}ynamical stability in
%morphoelasticity. {\it Bull. Math. Biol.},  81 (2019), no. 8, 3219--3244.
%\EEE




 \bibitem{fournier-et-al}
  M.~Fournier, H.~Baill{\`e}res, B.~Chanson. Tree biomechanics:
  growth, cumulative prestresses, and reorientations. {\it
    Biomimetics},  2 (1994), no. 3, 229--251. 


\bibitem{ganghoffer-et-al}
J-F.~Ganghoffer, P.~I.~Plotnikov, J.~Soko\l owski. Mathematical modeling of volumetric material growth in thermoelasticity. {\it J. Elasticity} 117 (2014), no. 1, 111--138. 


 \bibitem{goriely}
A.~Goriely. {\it The mathematics and mechanics of biological growth}.
Interdisciplinary Applied Mathematics, 45. Springer, New York, 2017.



\bibitem{hajlasz-koskela}
P.~Haj\l asz, P.~Koskela. Sobolev met Poincar\'e. {\it 
  Mem. Amer. Math. Soc.} 145 (2000), no. 688, x+101 pp. 


\bibitem{hodge}
N.~Hodge, P.~Papadopoulos. Continuum modeling and numerical simulation of cell motility. {\it J. Math. Biol.} 64 (2012), no. 7, 1253--1279. 

\bibitem{horn}
T.~Horn, O.~Harrysson. Overview of current additive manufacturing technologies and selected applications. {\it Science Progress,} 95 (2012), no. 3, 255--282. 


\bibitem{jiang}
R.~Jiang, A.~Kauranen. Korn's inequality and John
domains. {\it Calc. Var. Partial Differential Equations}, 56 (2017),
no. 4, Paper No. 109, 18 pp.


\bibitem{JohnF}
F.~John. Rotation and strain. {\it Comm. Pure Appl. Math.} 14 (1961), 391--413.


\bibitem{koskela-rohde}
  P.~Koskela, S.~ Rohde. Hausdorff dimension and mean porosity. {\it Math. Ann.} 309 (1997), no. 4, 593--609. 

\bibitem{langer}
J.~S.~Langer,
Instabilities and pattern forestation in crystal growth.
{\it Rev. Mod. Phys.} 52 (1980), no. 1, 1.
%\bibitem{Lorenz}
%T.~Lorenz.
%Boundary regularity of reachable sets of control systems. 
%{\it Systems Control Lett.} 54 (2005), no. 9, 919--924.


\bibitem{lychev}
S.~Lychev, A.~Manzhirov. The mathematical theory of growing bodies. Finite deformations. {\it J. Appl. Math. Mech.} 77 (2013), no. 4, 421--432. 




\bibitem{martio}
O.~Martio, J.~Sarvas. Injectivity theorems in plane and space. {\it Ann. Acad. Sci. Fenn. Ser. A I Math.} 4 (1979), no. 2, 383--401.


\bibitem{metlov}
V.~Metlov. On the accretion of inhomogeneous viscoelastic bodies under finite deformations. {\it J. Appl. Math. Mech.} 49 (1985), no. 4, 490--498.


%  \bibitem{murashkin-et-al}
% E.~V.~Murashkin, E.~P.~Dats, N.~Stadnik. Application of surface growth model for a pathological process in a blood vessel's wall.  {\it Math. Methods Appl. Sci.} 45 (2022), no.5,  3197--3212.  





\bibitem{naumov}
V.~E.~Naumov. Mechanics of growing deformable solids: a review. {\it J. Eng. Mech.} 120 (1994), no. 2, 207--220.



 \bibitem{schwerdtfeger-et-al}
K.~Schwerdtfeger, M.~Sato, K.-H.~Tacke. Stress formation in solidifying bodies. Solidification in a round continuous casting mold. {\it Metall. Mater. Trans.
B.} 29 (1998), 1057--1068.  


\bibitem{skalak}
R.~Skalak, A.~Hoger. Kinematics of surface growth. {\it J. Math. Biol.} 35 (1997), no. 8, 869--907. 


\bibitem{southwell}
R.~Southwell. {\it Introduction to the theory of elasticity for engineers and physicists.} Oxford University Press, Oxford, 1941. 




\bibitem{sozioy}
F.~Sozio, A.~Yavari. Nonlinear mechanics of surface growth for cylindrical and
spherical elastic bodies. {\it J. Mech. Phys. Solids} 98 (2017), 12--48. 



\bibitem{sozio}
F.~Sozio, A.~Yavari. Nonlinear mechanics of accretion. {\it J. Nonlinear Sci.} 29 (2019), no. 4, 1813--1863. 

\bibitem{thompson}
D.~Thompson. {\it On growth and forms: the complete revised edition.}  Dover, New York, 1992. 


\bibitem{tomassetti}
G.~Tomassetti, T.~Cohen, R.~Abeyaratne. Steady accretion of an elastic body on a hard spherical surface and the notion of a four-dimensional reference space. {\it J. Mech. Phys. Solids} 96 (2016), 333-352. 



\bibitem{truskinovsky}
L.~Truskinovsky, G.~Zurlo. Nonlinear elasticity of incompatible surface growth. {\it Phys. Rev. E} 99 (2019), no. 5, 053001. 


\bibitem{vaisala}
J. V\"ais\"al\"a. Exhaustions of John domains. {\it Ann. Acad. Sci. Fenn. Ser. A I Math.} 19 (1994), no. 1, 47--57.


%\bibitem{zhang-et-al}
%H.~Zhang, J.~P.~Tian, B.~Niu, Y.~Guo. Mathematical modeling of tumor surface growth with necrotic kernels. {\it Math. Methods Appl. Sci.} 44 (2021), no. 17, 12688--12706. 



\bibitem{zurlo}
G.~Zurlo. L.~Truskinovsky. Inelastic surface growth. {\it Mech. Res. Comm.}
93 (2018),   174--179.
\end{thebibliography}
\end{document}